\documentclass[11pt,letterpaper,reqno]{amsart}
\usepackage{amssymb,amsmath,amsthm,amsfonts}
\usepackage{tikz}
\usepackage{verbatim}

\newtheorem{theorem}{Theorem}
\newtheorem{lemma}[theorem]{Lemma}
\newtheorem{proposition}[theorem]{Proposition}

\newcommand{\R}{\mathbb{R}}
\newcommand{\F}{\boldsymbol{F}}
\newcommand{\Z}{\mathbb{Z}}
\newcommand{\C}{\mathcal{C}}
\renewcommand{\S}{\mathcal{S}}
\newcommand{\D}{\mathcal{D}}
\newcommand{\T}{\mathcal{T}}
\newcommand{\M}{\mathcal{M}}
\newcommand{\1}{\mathbf{1}}
\renewcommand{\L}{\mathrm{L}}

\numberwithin{equation}{section}

\begin{document}

\title{$\L^p$ estimates for a  singular entangled quadrilinear form }
\author{Polona Durcik}
\address{Mathematisches Institut, Universit\"at Bonn
 Endenicher Allee 60, 53115 Bonn, Germany}
\email{durcik@math.uni-bonn.de}
\date{\today}


\begin{abstract} We prove $\L^p$ estimates for a continuous version of a dyadic quadrilinear form  introduced by Kova\v{c} in \cite{kovac:tp}.  
This improves the range of exponents from the prequel  \cite{pd:L4} of the present paper.
\end{abstract}
\maketitle
\section{Introduction}
This article  is a continuation of \cite{pd:L4}. We are  concerned with a quadrilinear singular integral form   involving the    {\em entangled} product of four functions  on $\R^2$ 
\begin{align*} 
\F(F_1,F_2,F_3,F_4)(x,y,x',y'):=
F_1(x,y)F_2(x',y)F_3(x',y')F_4(x,y').
\end{align*} 
For Schwartz  functions $F_j\in \S (\R^2)$, the form   is given by
\begin{align*}
\Lambda(F_1,F_2,F_3,F_4):=  \int_{\R^2}  \widehat{\F} (\xi,\eta,-\xi,-\eta)m(\xi,\eta) d\xi d\eta, 
\end{align*}
where  $\F:=  {\F}(F_1,F_2,F_3,F_4)$ and    $m$  is a bounded function on $\R^2$,   smooth away from the origin. For all   multi-indices $\alpha$ up to some large finite order  it satisfies\footnote{We write $A\lesssim B$ if there is an absolute constant $C>0$ such that $A\leq C B$. If $P$ depends on a set of parameters $P$,   we write $A\lesssim_P B$. We write $A\sim B$ if both $A\lesssim B$ and $B\lesssim A$. }
\begin{align*}
 |\partial^\alpha m(\xi,\eta)|\lesssim (|\xi|+|\eta|)^{-|\alpha|}. 
\end{align*} 
In \cite{pd:L4} it is shown that
\begin{align}\label{def:main}
|\Lambda(F_1,F_2,F_3,F_4)|\lesssim \|F_1\|_{\L^4(\R^2)}\|F_2\|_{\L^4(\R^2)}\|F_3\|_{\L^4(\R^2)}\|F_4\|_{\L^4(\R^2)}.
\end{align}
 Our present goal is to prove $\L^p$ estimates for $\Lambda$ in a larger range of exponents.
\begin{theorem}\label{thm:main}
For $F_1,F_2,F_3,F_4\in \S(\R^2 )$, the quadrilinear form $\Lambda$ satisfies  
\begin{align*}
|\Lambda(F_1,F_2,F_3,F_4)|\lesssim_{(p_j)}  \|F_1\|_{\L^{p_1}(\R^2)}\|F_2\|_{\L^{p_2}(\R^2)}\|F_3\|_{\L^{p_3}(\R^2)}\|F_4\|_{\L^{p_4}(\R^2)}
\end{align*}
whenever $\sum_{j=1}^4\frac{1}{p_j}=1$ and $2<p_j \leq \infty$ for all $j$.
\end{theorem}
This theorem  is a consequence of  the    restricted type estimates   given by Theorem \ref{thm:GRTmain} below. 
By the decomposition performed in \cite{pd:L4},    it suffices to prove Theorem \ref{thm:main} for  $m$ reduced to a single cone in the frequency plane $(\xi,\eta)$. More precisely, it is enough to consider
the form 
  \begin{align}\label{form:enough-to-consider}
 \int_0^\infty   {\mu}_t  \int_{\R^2}  \widehat{\boldsymbol{F}}(\xi,\eta,-\xi,-\eta) \widehat{\varphi^{(u)}}(t\xi)  \widehat{\psi^{(v)}}(t\eta)\widehat{\varphi^{(-u)}}(-t\xi)\widehat{\psi^{(-v)}}(-t\eta)d\xi d\eta \frac{dt}{t}  
\end{align}
 where  $\varphi^{(u)}(x)=(1+|u|)^{-25}\varphi(x-u)$ and $\psi^{(v)}(x)=(1+|v|)^{-10}\psi(x-v)$. The functions   
 $\varphi,\psi \in \S(\R)$ are real-valued and $\psi$ is such that  $(\int_\eta^\infty |\widehat{\psi}(\tau)|^2    d\tau/\tau )^{1/2}$  belongs to $\S(\R)$,  $u,v\in\R$ and ${\mu}_t$ 
are measurable coefficients  with   $|\mu_t|\leq 1$. 
We remark that  the decomposition is  not explicitly stated in this manner  in \cite{pd:L4}, but it follows by a minor rephrasing of the arguments. The estimate for \eqref{form:enough-to-consider} will  be uniform in the parameters $u,v$.

Since the integral of the Fourier transform of a Schwartz function  over a hyperplane in $\R^4$ equals the integral of  the function itself over the perpendicular hyperplane, we  can express
 the form \eqref{form:enough-to-consider}  as 
\begin{align*} 
  \int_0^\infty {\mu}_t   \int_{\R^2}   {\boldsymbol{F}} *
[{\varphi^{(u)}}\otimes {\psi^{(v)}} \otimes {\varphi^{(-u)}} \otimes{\psi^{(-v)}}]_t(p,q,p,q)dpdq \frac{dt}{t} ,
\end{align*}
where $(f_1\otimes\dots \otimes f_n)(x_1,\dots x_n):=f_1(x_1)\dots f_n(x_n)$ and $[f]_t(x_1,\dots,x_n):=t^{-n}f(t^{-1}x)$.  We   truncate in the scale $t$, that is, for $N>0$ we consider     $\Lambda_{\varphi,\psi}^N = \Lambda_{\varphi,\psi,\mu,u,v}^N$ given by
\begin{align*} 
\Lambda_{\varphi,\psi}^N(F_1,F_2,F_3,F_4):= \int_{2^{-N}}^{2^N} {\mu}_t   \int_{\R^2}   {\boldsymbol{F}} *
[{\varphi^{(u)}}\otimes {\psi^{(v)}} \otimes {\varphi^{(-u)}} \otimes{\psi^{(-v)}}]_t(p,q,p,q)dpdq \frac{dt}{t}, 
\end{align*}
which is well defined for bounded measurable functions $F_j$ with finite measure support.  
We have the following analogue of Theorem \ref{thm:main} for  $\Lambda_{\varphi,\psi}^{N}$. 
\begin{theorem}
\label{thm:truncated}
For bounded measurable functions $F_1,F_2,F_3,F_4$ with finite measure support, the quadrilinear form $\Lambda_{\varphi,\psi}^{N}$ satisfies the estimate 
\begin{align}\label{est:lambda-eps-N}
|\Lambda_{\varphi,\psi}^{N}(F_1,F_2,F_3,F_4)|\lesssim_{(p_j)}  \|F_1\|_{\L^{p_1}(\R^2)}\|F_2\|_{\L^{p_2}(\R^2)}\|F_3\|_{\L^{p_3}(\R^2)}\|F_4\|_{\L^{p_4}(\R^2)}
\end{align}
whenever $\sum_{j=1}^4\frac{1}{p_j}=1$ and $2<p_j \leq \infty$ for all $j$.
\end{theorem}
The bound \eqref{est:lambda-eps-N}   is independent of $N,u,v$. Approximating $F_j\in \S$  in $\L^{p_j}$ with smooth compactly supported functions, Theorem \ref{thm:truncated}  then implies Theorem \ref{thm:main}.  
By the multilinear interpolation and the restricted type theory discussed in \cite{th:wpa}, Theorem \ref{thm:truncated} is a   consequence of the following   (generalized) restricted type estimates. 
\begin{theorem}\label{thm:GRTmain}   For $j=1,2,3,4,$ let $E_j \subseteq \R^2$ be a set of finite measure. Let  $k$ be the largest index such that $|E_k|$ is maximal among the $|E_j|$.  Then there exists a subset 
 ${E}'_k\subseteq E_k$ with $2|{E}'_k| \geq |E_k|$, such that 
for any four measurable functions $F_j$ with\footnote{By $\1_{A}$ we denote the characteristic function of a set $A\subseteq  \R^2$. } $|F_j|\leq \1_{E_j}$ for all $j$ and $|F_k|\leq \1_{{E}'_k}$ we have  the estimate
\begin{align*}
|\Lambda_{\varphi,\psi}^{N} (F_1,F_2,F_3,F_4)|\lesssim |E_1|^{\alpha_1}|E_2|^{\alpha_2}|E_3|^{\alpha_3}|E_4|^{\alpha_4}
\end{align*}
whenever $\sum_{j=1}^4 \alpha_j=1$ and  $ -1/2 \leq  \alpha_j \leq 1/2$ for all $j$.
\end{theorem}
 Negative exponents $\alpha_j$ correspond to  quasi-Banach space estimates for the dual operators of  $\Lambda^N_{\varphi,\psi}$, for which one may consult  \cite{th:wpa}.   

Assuming Theorem \ref{thm:main}, we now mention  how to extend $\Lambda$ to a bounded operator on $\L^{p_1}\times \L^{p_2}\times \L^{p_3}\times \L^{p_4}$ whenever $p_j$ are  as in Theorem \ref{thm:main}. 
If $p_j<\infty$ for all $j$, this follows by density of $\S$ in $\L^{p_j}$. If $p_j=\infty$ for some $j$, we argue by duality. Note that   have at most one exponent equal to $\infty$. We sketch the argument when  $p_4=\infty$, the other instances following by   symmetry of the form.
We know that there is an operator $T$ mapping $\L^4 \times \L^4 \times \L^4$ to $\L^{4/3}$ such that
$$\Lambda(F_1,F_2,F_3,F_4)=\int T(F_1,F_2,F_3)F_4.$$ 
We claim that for $F_j\in \S$,
 $\|T(F_1,F_2,F_3)\|_{\L^1}\lesssim \|F_1\|_{\L^{p_1}}\|F_2\|_{\L^{p_2}}\|F_3\|_{\L^{p_3}}$. Then $\Lambda$ can be defined on $\S\times \S \times \S \times \L^\infty$  and   density arguments   yield  a bounded  extension on $\L^{p_1}\times \L^{p_2} \times \L^{p_3} \times \L^{\infty}$. 
To see the claim  we   write
$$\|T(F_1,F_2,F_3)\|_{\L^1([-M,M]^2)} =  \int  T(F_1,F_2,F_3)\vartheta$$
where $\vartheta$   is a modulation times $\1_{[-M,M]^2}$. Then we approximate  $\vartheta$ weakly in $\L^4$ with smooth compactly supported functions having   $\L^\infty$ norms uniformly bounded by $1$. Applying  Theorem \ref{thm:main} for the tuple  $(p_1,p_2,p_3,\infty)$   yields the assertion.   

Let us briefly comment on the form
  $\Lambda$. For more extensive motivation  we refer to \cite{pd:L4}. The  instance of $\Lambda$ which was first considered is  the trilinear form\footnote{In \cite{pd:L4} we called this form $T$, not to be interchanged  with the dual operator introduced above.}   $\Lambda_1(F_1,F_2,F_3):=\Lambda(F_1,F_2,F_3,1)$. It was   introduced by Demeter and Thiele \cite{demTh:BHT}. 
This trilinear form can also be seen as a simpler version of the twisted paraproduct     proposed by Camil Muscalu and sometimes one refers to it with that name as well. 

Boundedness of  $\Lambda_1$ was established by Kova\v{c} \cite{kovac:tp},  who first investigated a dyadic model   of   $\Lambda$ for a general function  $F_4$  by an induction on scales type argument. See also \cite{kovac:bf}. This led to an estimate for a dyadic version of  $\Lambda_1$ whenever $2<p_1,p_2,p_3<\infty$ and $1/{p_1}+1/{p_2}+1/{p_3}=1$. Then Kova\v{c}  passed to the bound for $\Lambda_1$  using the square functions of Jones, Seeger and Wright \cite{jsw}.  Bernicot's fiber-wise  Calder\'on-Zygmund decomposition \cite{bernicot:fw}    extended the  
    range of exponents  to $1<p_1,p_3<\infty$, $2<p_2\leq \infty$. The transition  to the continuous case and the extension of the exponent range  both relied   on the special structure  arising from   $F_4=1$.  

For the quadrilinear form  with   a general fourth function,  the   $\L^4$ estimate  \eqref{def:main}  was derived by adapting the induction  of scales techique by Kova\v{c} to the continuous setting. In the present    article we prove estimates in a larger range of exponents by  extending his method to   the continuous localized context.
 
By  a   classical stopping time argument,   Theorem \ref{thm:GRTmain} is reduced to estimating entangled    forms    of the type
\begin{align*}
  \int_{\Omega }  |   {\boldsymbol{F}} *
[{\varphi^{(u)}}\otimes {\psi^{(v)}} \otimes {\varphi^{(-u)}} \otimes{\psi^{(-v)}}]_t(p,q,p,q)|dpdq \frac{dt}{t}.
\end{align*} 
Here  $\Omega$ is a certain  local    region in the upper half space  with "regular" boundary.  Controlling such  objects with the technique from \cite{kovac:tp} requires an algebraic telescoping identity. In \cite{pd:L4},  its derivation   relies on an identity involving  the Fourier transform. The argument is of global nature and we cannot directly repeat it in the localized setting.

We obtain the desired telescoping element in  Proposition \ref{prop:tel} in Section \ref{sec:tel}. 
To  overcome the mentioned difficulty,  we  first restrict the functions $F_j$ to certain projections of the region $\Omega$.
This allows us to discard the spatial localization of the form and proceed in the manner of \cite{pd:L4}.  The issue in the described  process is  then in estimating    boundary terms,  representing   differences between   local and global objects. This requires certain control of the boundary and is carried out in  Lemma \ref{lemma:error-Tc} and Lemma \ref{lemma:bdry-term} below.
Our approach has been inspired by Muscalu, Tao and Thiele \cite{mmt:uniform}.

To conclude we remark that in general  we  do not know of any arguments which could extend the   range of exponents from Theorem \ref{thm:main} to $p_j\leq 2$.\\

  {\bf Acknowledgement.} {I would like to express my sincere  gratitude to  my advisor Prof. Christoph Thiele  for his guidance and support throughout this project.}


\section{Local telescoping} 
\label{sec:tel}
First let us  set up some notation. 
A {\em dyadic interval}   is a interval of the form $[2^km, 2^k(m+1)]$ for some $k,m\in \Z$.  We denote the set of all dyadic intervals by $\mathcal{I}$ and the set of all dyadic intervals of length $2^k$ by $\mathcal{I}_k$.
A {\em dyadic square} is the Cartesian product of two dyadic intervals of the same length. For a dyadic square $S$ we denote by $\ell(S)$ its sidelength.  We write $\D$ for the set of all dyadic squares and $\D_k$   for the set of all dyadic squares of sidelength $2^k$. 
Each   $S\in \mathcal{D}$ is divided into four congruent dyadic squares of half the sidelength, called the {\em children} of $S$. Conversely, each square in $\mathcal{D}$ has a unique {\em parent} in $\mathcal{D}$.  Given any two dyadic squares, either one is contained in the other or they are almost disjoint, by which we  mean that their intersection   has Lebesgue measure zero.  

As in \cite{kovac:tp}, we collect the squares   into units called {\em trees}. A finite collection   $\mathcal{\T}\subseteq \mathcal{D}$ is called  a {\em tree} if there exists a 
square $R_{\T}\in \mathcal{\T}$ called the {\em root}, satisfying $S\subseteq R_{\T}$ for every $S\in \mathcal{\T}$.  A tree is called {\em convex} if for all $S_1,\,S_2,\,S_3$ we have that $S_1\subseteq S_2\subseteq S_3$ and $S_1,\; S_3\in \mathcal{\T}$ imply $S_2\in \mathcal{\T}$. A {\em leaf} of $\mathcal{\T}$ is a dyadic square which is not contained in $\mathcal{\T}$, but its parent is. We denote the set of leaves of $\T$ by $\mathcal{L}(\T)$. 
Note that the leaves of a   convex tree   partition its root.
We split $\T$ into generations of squares of sidelength  $2^k$. For this we denote 
\begin{align*}
\mathcal{T}_k := \T\cap \D_k\;\;\;\mathrm{and}\;\;\;\T_k^c:=\D_k\setminus \T_k.
\end{align*}  
For the union of all squares in $\T_k$ we write
$$T_k:=  \bigcup_{S\in \T_k}S.$$
Observe that for a convex tree $\mathcal{\T}$ we have $T_k \subseteq T_{k'}$ if $k \leq k'$, $T_{k'}\neq \emptyset$.
The following lemma measures the "size" of the boundary of $T_k$. It estimates the cardinality  of dyadic points  
$$\Delta(\T_k):= \partial T_k \cap (2^k\Z\times 2^k\Z).$$
This is a variant of Lemma 4.8 from  \cite{mmt:uniform}. 
\begin{lemma} \label{lemma:bdry}
For any   convex tree $\mathcal{\T}$ we have
\begin{align*}
\sum_{k\in \Z}2^{2k}\#\Delta(\T_k)  \lesssim |R_{\T}|.
\end{align*}
\end{lemma}
\begin{proof}
It suffices to prove the claim for all dyadic points $(p,q) \in \partial T_k$ such that $[p-2^k,p]\times [q-2^k,q] \not\in \T_k$. 
 For each such point consider  the dyadic square 
\begin{align*}
S(p,q,k):=  [p-2^k, p-2^{k-1}] \times  [q-2^k, q-2^{k-1}]
\end{align*}
  which has   area $2^{2(k-1)}$. We claim that   squares of this form are pairwise almost disjoint. This will prove the lemma, as they are contained in $3R_{\T}$.

To see the claim, suppose that  $S({p,q,k})$ and $S(p',q',k')$ intersect in a set of positive measure. If $k=k'$, then they must coincide since they are dyadic and of the same scale.
So suppose that $k<k'$,   hence  $S(p,q,k)$ is contained in $S(p',q',k')$. Then the point $(p,q)$ is contained in the  interior of  $[p'-2^{k'},p']\times [q'-2^{k'}, q']$, which is disjoint from $T_{k'}$.  This shows that $(p,q)\in T_k$  but $(p,q)\notin T_{k'}$,
contradicting   convexity of $\T$.
\end{proof}
 
With any  collection of dyadic squares $\C\subseteq \D$ we associate a region in the upper half space $\R^3_+$. The region consists of Whitney  boxes associated with $S\in \C$ and is defined by
\begin{align*}
\Omega_{\C}:=\bigcup_{S\in \C}  {S} \times \Big [\frac {\ell(S)}{2},\ell(S) \Big ].
\end{align*}
The case $\C=\T$ for a   convex tree $\T$ is depicted in Figure \ref{fig:tree}. 
 Observe that $\Omega_{\T}=\cup_{k\in \Z}\Omega_{\T_k} =\cup_{k\in \Z}T_k\times [2^{k-1}, 2^k].$
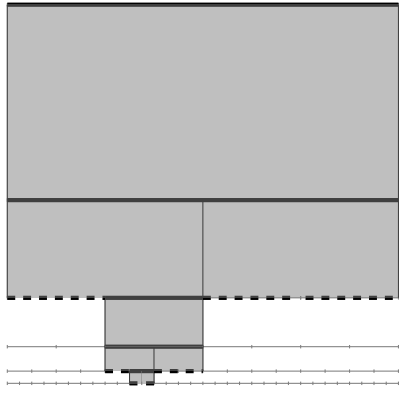
\begin{figure}[htb] 
\begin{tikzpicture}[scale=0.65]
\draw (0,8)-- (8,8);
\draw[gray] (4,2)-- (8,2);
\draw[gray] (0,1)-- (2,1);
\draw[gray] (4,1)-- (8,1);
\draw[gray] (0,0.5)-- (2,0.5);
\draw[gray] (3,0.5)-- (8,0.5);
\draw[gray] (0,0.25)-- (2.5,0.25);
\draw[gray] (3,0.25)-- (8,0.25);
\draw[gray] (0cm, 8cm-1pt) -- (0cm, 8cm+1pt);
\draw[gray] (8cm, 8cm-1pt) -- (8cm, 8cm+1pt);
\draw[gray] (0cm, 4cm-1pt) -- (0cm, 4cm+1pt);
\draw[gray] (4cm, 4cm-1pt) -- (4cm, 4cm+1pt);
\draw[gray] (8cm, 4cm-1pt) -- (8cm, 4cm+1pt);
\foreach \a in {0,2,4,6,8}
\draw[gray] (\a cm, 2cm-1pt) -- (\a cm, 2cm+1pt);
\foreach \x in {0,1,2,3,4,5,6,7,8}
\draw[gray] (\x cm, 1cm-1pt) -- (\x cm, 1cm+1pt);
\foreach \y in {0,0.5,1,1.5,2,2.5,3,3.5,4,4.5,5,5.5,6,6.5,7,7.5,8}
\draw[gray] (\y cm, 0.5cm-1pt) -- (\y cm,0.5cm+1pt);
\foreach \z in {0,0.5,1,1.5,2,2.5,3,3.5,4,4.5,5,5.5,6,6.5,7,7.5,8}
\draw[gray] (\z cm, 0.25cm-1pt) -- (\z cm,0.25cm+1pt);
\foreach \w in {0.5,1,1.5,2,2.5,3,3.5,4,4.5,5,5.5,6,6.5,7,7.5,8}
\draw[gray] (\w cm-0.25cm, 0.25cm-1pt) -- (\w cm-0.25cm,0.25cm+1pt);

\draw[gray] (6cm, 2cm-1pt) -- (6cm, 2cm+1pt);
\draw[ultra thick,black] (0,8)-- (8,8);
\draw[ultra thick,black] (0,4)-- (8,4);
\draw[ultra thick,black] (2,2)-- (4,2);
\draw[ultra thick,black] (2,1)-- (3,1);
\draw[ultra thick,black] (2.5,0.5)-- (3,0.5);
\draw[ultra thick,black] (3,1)-- (4,1);
\draw[dashed, ultra thick,black] (4,2)-- (5.9,2);
\draw[dashed, ultra thick,black] (6.1,2)-- (8,2);
\draw[dashed, ultra thick,black] (0,2)-- (2,2);
\draw[dashed, ultra thick,black] (2,0.5)-- (2.5,0.5);
\draw[dashed, ultra thick,black] (3,0.5)-- (4,0.5);
\draw[dashed, ultra thick,black] (2.5,0.25)-- (2.75,0.25);
\draw[dashed, ultra thick,black] (3,0.25)-- (2.75,0.25);

\draw[black] (0,8)--(0,2);
\draw[black] (8,8)--(8,2);
\draw[black] (4,4)--(4,0.5);
\draw[black] (2,2)--(2,0.5);
\draw[black] (2.5,0.5)--(2.5,0.25);
\draw[black] (3,1)--(3,0.25);


\fill[gray, opacity = 0.5](0,8) -- (0,4)-- (8,4) -- (8,8) --cycle;
\fill[gray, opacity = 0.5](0,4)-- (8,4) -- (8,2) -- (0,2)--cycle;
\fill[gray, opacity = 0.5](2,2) -- (4,2) -- (4,1)--(2,1)--cycle;
\fill[gray, opacity = 0.5] (3,1)--(2,1)--(2,0.5)--(3,0.5)--cycle;
\fill[gray, opacity = 0.5] (3,1)--(4,1)--(4,0.5)--(3,0.5)--cycle;
\fill[gray, opacity = 0.5] (2.75,0.5)--(3,0.5)--(3,0.25)--(2.75,0.25) --cycle;
\fill[gray, opacity = 0.5] (2.75,0.5)--(2.5,0.5)--(2.5,0.25)--(2.75,0.25) --cycle;
\end{tikzpicture}
\caption{Projection of     $\Omega_\T$ on  $\R^2_+$. The bold lines represent $S\times \ell(S)$ for   $S\in \T$, while the dotted lines correspond to $S\in \mathcal{L}(\T)$.} \label{fig:tree}
\end{figure}

Throughout the text, all two-dimensional   functions     will be measurable, bounded, with finite measure support  and positive. 
Denote \begin{align*}
\theta(x,y):=(1+|(x,y)|^4)^{-1} .
\end{align*}
For   a function $F$ on $\R^2$ and $\C\subseteq \D$  we define
\begin{align*}
M(F,\C):=\sup_{(p,q,t)\in  {\Omega_\C}} (F^2*[\theta]_t(p,q))^{1/2}.
\end{align*} 
Denote also 
$$\vartheta(x):=(1+|x|)^{-4}.$$

 Now we consider   a continuous variant of the Gowers box inner product  used in \cite{kovac:tp}. The  following  estimate joins a 
  version of   the   box Cauchy-Schwarz inequality and   an estimate of the Gowers box norm by an $\L^2$-type average.  This is the reason for the restricted range of exponents in Theorem \ref{thm:main}. 

\begin{lemma}\label{lemma:a-size}
For $(p,q,t) \in  \Omega_\C$ we have
\begin{align}\label{gowers-bip}
\F * [\vartheta\otimes \vartheta\otimes\vartheta\otimes\vartheta]_t(p,q,p,q)   \leq \prod_{j=1}^4 M(F_j,\C).
\end{align}
\end{lemma}
\begin{proof} 
Denote the left-hand side of \eqref{gowers-bip} by  $A^{(p,q,t)}(F_1,F_2,F_3,F_4)$  and rewrite it as
\begin{align}\label{gowers:cs}
\int_{\R^2}    \Big( \int_\R F_1(x,y)F_2(x',y)  [\vartheta]_t(q-y) dy \Big ) 
\Big( \int_\R   F_3(x',y')F_4(x,y') [\vartheta]_t(q-y') dy' & \Big )   \nonumber \\
 [\vartheta]_t(p-x)[\vartheta]_t(p-x') dxdx' &  
\end{align}
Now we  apply the   Cauchy-Schwarz inequality with respect to  $[\vartheta]_t(p-x)dx,[\vartheta]_t(p-x')dx'$, which bounds this term   by
\begin{align*}
A^{(p,q,t)}(F_1,F_2,F_2,F_1)^{1/2}A^{(p,q,t)}(F_4,F_3,F_3,F_4)^{1/2}.
\end{align*}
By symmetry in $(p,q)$ it   follows   that
\begin{align*}
A^{(p,q,t)}(F_1,F_2,F_2,F_1) \leq A^{(p,q,t)}(F_1,F_1,F_1,F_1)^{1/2}A^{(p,q,t)}(F_2,F_2,F_2,F_2)^{1/2}.
\end{align*}

Now we write $A^{(p,q,t)}(F_j,F_j,F_j,F_j)$ in the same way as in \eqref{gowers:cs} and apply   the   Cauchy-Schwarz inequality with respect to $dy,dy'$. This yields     
\begin{align*}
A^{(p,q,t)}(F_j,F_j,F_j,F_j)\leq (F_j^2*[\vartheta\otimes\vartheta]_t(p,q))^2 
\leq (F_j^2*[\theta]_t(p,q))^2,
\end{align*}
which proves the claim.
\end{proof}

 With functions $\phi_j\in \L^1(\R)$, $j=1,2,3,4,$ and $\C\subseteq \D$ we associate the  local form    
\begin{align*}
\Theta^{\C}_{\phi_1,\phi_2,\phi_3,\phi_4}(F_1,F_2,F_3,F_4) & :=
 \int_{\Omega_\C}   \F  * [\phi_1\otimes \phi_2\otimes\phi_3\otimes\phi_4]_t(p,q,p,q) dpdq \frac{dt}{t}. 
\end{align*}
To shorten the notation we write  $\Theta^\C_{\phi_1,\phi_3} :=\Theta^\C_{\phi_1,\phi_3,\phi_1,\phi_3}$.

The following two  complementary lemmas will be used to control error and boundary terms in Proposition \ref{prop:tel}.
\begin{lemma}\label{lemma:error-Tc} 
For a convex tree $\T$  we have
\begin{align}\label{error}
  \sum_{k\in \Z} \Theta^{\T_k}_{\vartheta^2,\vartheta^2}(F_1\1_{T_k^c},F_2,F_3,F_4)   \lesssim |R_\T| \prod_{j=1}^4M(F_j,\T).
 \end{align}
\end{lemma}
Observe that by symmetry of  \eqref{error},   the same result holds under any permutation of the arguments  $F_1\1_{T_k^c},F_2,F_3,F_4$. 
\begin{proof}
For $k\in \Z$ and $t\in [2^{k-1},2^k]$ we consider    
\begin{align}\label{error-form-in-question} 
\int_{T_k} \int_{\R^4}  \F(F_1\1_{T_k^c},F_2,F_3,F_4)   (x,y,x',y')    [\vartheta  \otimes \vartheta \otimes \vartheta  \otimes \vartheta ]_t(p-x,q-y,p-x',q-y') & \nonumber \\
 \vartheta  \otimes \vartheta  (t^{-1}(p-x,q-y))  \vartheta  \otimes \vartheta   (t^{-1}(p-x',q-y'))  dxdydx'dy' dpdq. &  
\end{align} 
Note that \eqref{error} is obtained 
by integrating this term in $t\in [2^{k-1},2^k]$ and summing over $k\in \Z$.
We claim that for  $(x,y)\in T_k^c$ and $(p,q)\in  T_k$   
  there is a point $(a,b)$   contained in
\begin{align} \label{b-pq}
B(p,q):= \{(p',q') \in \partial T_k : p'=p\;\mathrm{or}\;q'=q \}   \cup \Delta(\T_k)
\end{align}
 such that
$|(p,q)-(x,y)|\geq   |(p,q)-(a,b)|.$ 

This can be seen as follows. By  $E$ we  denote  the intersection of $\partial T_k$ and the line segment between  $(p,q)$ and $(x,y)$. If $E$ contains dyadic points from $\Delta(\T_k)$, we may  set $(a,b)$ to be any of these points. 
Otherwise, $E$  must contain a point of the form $(p', q'+\alpha)$ or $(p'+\alpha, q')$ for some $p',q'\in 2^k\Z,\,\alpha\in (0,2^k)$. 
Assume it contains at least  one  of the form $(p', q'+\alpha)$. For definiteness    pick  the one  with the   the least distance to $(p,q)$. 
In case  $q'<q<q'+2^k$ we know that $(p',q)\in \partial T_k $ and we set $(a,b)=(p',q)$. 
If $q< q' $, we set $(a,b)=(p', q') \in \Delta(\T_k)$. In case $q> q'+2^k $  we choose $(p', q'+2^k) \in \Delta(\T_k)$.
Analogously we proceed in the remaining case, that is, if $E$  consists only of   points  $(p'+\alpha,q')$.

Since $\vartheta\otimes \vartheta \leq \theta$ and $\theta$ is radially decreasing, we   have for $(p,q),(x,y),(a,b)$ as above
$$\vartheta\otimes \vartheta(t^{-1}(p-x,q-y)) \leq \theta(t^{-1}(p-a,q-b)) \leq  \sum_{(a,b)\in B(p,q)} \theta(t^{-1}(p-a,q-b)).$$ 
Estimating   $\vartheta  \otimes \vartheta   (t^{-1}(p-x',q-y')) \leq 1$,   the term 
\eqref{error-form-in-question} is bounded by
\begin{align*} 
\nonumber
  \int_{T_k}    \F(F_1\1_{T_k^c},F_2,F_3,F_4) *   [\vartheta\otimes \vartheta\otimes \vartheta \otimes   \vartheta]_t(p,q,p,q)  
 \sum_{(a,b)\in B(p,q)} \theta (t^{-1}(p-a,q-b))  dpdq. 
\end{align*} 
Applying Lemma \ref{lemma:a-size},  the last display is no greater than
$$\Big( M(F_1\1_{T_k^c},\T_k)\prod_{j=2}^4 M(F_j,\T_k) \Big)    \int_{T_k}  \sum_{(a,b)\in B(p,q)} \theta (t^{-1}(p-a,q-b))  dpdq.  $$
Observe  that by homogeneity of the inequality \eqref{error} we may assume   $M(F_j,\T)=1$ for all $j$.  Due to this fact and  by 
    symmetry  in $p,q$,  it suffices  to   further estimate  
\begin{align*}
 \sum_{Q\in \mathcal{I}_k} \int_Q    \sum_{a:\{a\}\times Q \subseteq \partial T_k} \int_{\R } \theta(t^{-1}(p-a,0)) dpdq
  +  \sum_{(a,b)\in \Delta(\T_k)} \int_{\R^2}  \theta(t^{-1}(p-a,q-b)) dpdq.
\end{align*} 
Integrating the function $\theta$, the last display is       estimated  by  a constant times
$$\sum_{Q\in \mathcal{I}_k} \int_Q t\, \#\{a: \{a\}\times Q \subseteq \partial T_k\} +  t^2  \#\Delta(\T_k)
  \lesssim 2^{2k }  \#\Delta(\T_k).
 $$
Therefore,  up to a constant, \eqref{error}  is  bounded by 
 \begin{align*}
\sum_{k\in \Z} \int_{2^{k-1}}^{2^k}  2^{2k }  \#\Delta(\T_k) \frac{dt}{t}\lesssim  \sum_{k\in \Z}    2^{2k}  \#\Delta(\T_k)\lesssim |\mathcal{R}_\T|, 
 \end{align*}
which  is the desired result in view of the normalization $M(F_j,\T)=1$.
The last inequality follows from Lemma \ref{lemma:bdry}. 
\end{proof}
\begin{lemma}\label{lemma:bdry-term}
For a convex tree $\T$  we have
\begin{align}\label{bdry}
  \sum_{k\in \Z}     \Theta^{\T_k^c}_{\vartheta^2,\vartheta^2}(F_1\1_{T_k},F_2\1_{T_k},F_3\1_{T_k},F_4\1_{T_k})    \lesssim |R_\T| \prod_{j=1}^4M(F_j,\T).
 \end{align}
\end{lemma}
\begin{proof}
Proceeding in the exact same way as in the proof of Lemma  \ref{lemma:error-Tc} we see that the left-hand side of \eqref{bdry} is bounded by 
\begin{align*}
& \sum_{k\in \Z}     \Big( \prod_{j=1}^4 M(F_j\1_{T_k},\T^c_k)   \Big )   \int_{T_k^c} \sum_{(a,b)\in B(p,q)} \theta (t^{-1}(p-a,q-b))  dpdq\\
   \lesssim &  \sum_{k\in \Z}    \Big(  \prod_{j=1}^4 M(F_j\1_{T_k},\T^c_k)\Big )      \, 2^{2k} \#\Delta(\T_k)   ,
\end{align*}
where $B(p,q)$ is defined   as in \eqref{b-pq}.
We claim that  for each $j$ we have   
\begin{align*}
M(F_j\1_{T_k}, \T_k^c ) \lesssim M(F_j,\T).
\end{align*} 
Together with an application of Lemma \ref{lemma:bdry} this will finish the proof. 

The  claim can be rephrased as follows: for each $(p,q) \in T_k^c$ we have 
$$(F_j^2\1_{T_k}*[\theta]_t(p,q))^{1/2} \lesssim M(F_j,\T).$$
First we set
  $(p,q)=0$ without loss of generality. Also, we may assume that $T_k$ is contained in the   quadrant
$\{(p,q): p \geq 0,\, q\geq 0\}$,
as otherwise we   restrict $T_k$   to   each of the four quadrants  and  all   parts are   treated in the same way.   Denote 
$$r:=\min_{(a,b) \in \partial T_k} |(a,b)|.$$ 
Take any point $(a,b)$ which minimizes the distance and consider
 the closed cone $C$  in $\R^2$ with vertex $0$ and aperture $\pi/2$, its axis  being  the line   spanned by $(a,b)$. 
Observe that   each   $(x,y)\in T_k\cap C$ satisfies
 $|(x,y)| \geq |(x,y)-(a,b)|$ and thus $\theta(x,y) \leq \theta (x-a,y-b)$.  
 If $T_k\setminus C\neq \emptyset$, then we iterate 
 with $T_k$ replaced by ${T_k\setminus C}$. We find a point  $(a',b')\in  \partial T_k \cap {\partial(T_k\setminus C)}$ and a   cone $C'$   such that for 
each $(x,y)\in ({T_k\setminus C})\cap C'$ we have  $|(x,y)| \geq |(x,y)-(a',b')|$ and so $\theta(x,y) \leq \theta (x-a',y-b')$. 
Since $C\cup C'$   covers  $T_k$,   for each $(x,y)\in T_k$ we have
 $$ \theta(x,y) \leq \theta(a-x,b-y) + \theta( a'-x,b'-y).$$ 
Therefore, 
 \begin{align*}
 (F_j^2\1_{T_k}*[\theta]_t(0))^{1/2} \lesssim  (\sup_{(a,b,t)\in  {\Omega_{\T_k}}} F_j^2\1_{T_k}*[\theta]_t(a,b))^{1/2} \leq  \sup_{(a,b,t)\in  {\Omega_{\T}}} (F_j^2*[\theta]_t(a,b))^{1/2}
 \end{align*}
 as desired. 
\end{proof}
 For a function $f \in \S(\R)$ 
we consider the    Schwartz seminorm 
 $$\|f\|:=\sup_{x\in \R} \,(1+|x|)^8|f(x)|   +(1+|x|)^9 |f'(x)|.$$ 
 Now we are ready to state the estimate which will take the place of   the telescoping identities used in \cite{kovac:tp}, \cite{pd:L4}. 
\begin{proposition}\label{prop:tel}
Let $(\rho_i,\sigma_i)$ be two pairs of real-valued Schwartz functions which satisfy
\begin{align}\label{eqn:FT}
-t\partial_t|\widehat{\rho_i}(t\tau)|^2=|\widehat{\sigma_i}(t\tau)|^2.
\end{align}
Then  we have for any  convex tree $\T$
\begin{align}\label{identity-telescoping}
\Theta_{\rho_1,\sigma_2}^{\T}(F_1,F_2,F_3,F_4) +  \Theta_{\sigma_1,\rho_2}^{\T}(F_1,F_2,F_3,F_4) 
\lesssim_c |R_{\T}|\prod_{j=1}^4M(F_j,\T),
\end{align}
where   $c= \|\rho_1\|^2 \|\sigma_2\|^2 + \|\sigma_1\|^2\|\rho_2\|^2 + \|\rho_1\|^2 \|\rho_2\|^2$.
\end{proposition}

\begin{proof}
By homogeneity  of \eqref{identity-telescoping} we may assume $M(F_j,\T)=1$ for all $j$. By scaling invariance we may suppose   $|R_{\mathcal{\T}}|=1$. Thus, we are set to  establish
\begin{align} \label{eqn:tel-to-establish}
\Theta_{\rho_1,\sigma_2}^{\T}(F_1,F_2,F_3,F_4) +  \Theta_{\sigma_1,\rho_2}^{\T}(F_1,F_2,F_3,F_4) 
\lesssim_c 1.
\end{align}
Denote $\Psi:=\rho_{1}\otimes \rho_{2} \otimes \rho_{1} \otimes \rho_{2}$.
By the fundamental theorem of calculus  we have 
\begin{align}\label{eqn:ftc}
&  [\Psi]_{2^{k-1}} 
-   [\Psi]_{2^{k}} = 
\int_{2^{k-1}}^{2^k}   (-t\partial_t   
[\Psi]_t) \frac{dt}{t}.
\end{align}
We convolve the equality \eqref{eqn:ftc} with $\F $
and evaluate the convolution at $(p,q,p,q)$. Then we    integrate in $(p,q)$ over $T_k$ and sum over $k\in \Z$. Writing $T_k$ as the almost disjoint union of $S\in \T_k$, the   left-hand side of \eqref{eqn:ftc} becomes
\begin{align*}
L:= \sum_{k\in \Z}\sum_{S\in \T_k} \Big (  \sum_{S' \,\mathrm{child\,of } \,   S}\int_{S'} \F *[\Psi]_{\ell(S')}(p,q,p,q) dp dq 
-
  \int_{S} \F *[\Psi]_{\ell(S)}(p,q,p,q) dp dq  \Big )
\end{align*}
Since $\T$ is convex, each square $S\in \T\setminus \{R_\T\}$ has all four children $S'$ in $\T\cup \mathcal{L}(\T)$.  Thus, the  last display is a telescoping sum which equals  
\begin{align*}  
 \sum_{S\in \mathcal{L}(\T)}\int_{S} \F *[\Psi]_{\ell(S)}(p,q,p,q) dp dq 
-  
\int_{R_{\T}} \F *[\Psi]_{\ell(R_{\T})}(p,q,p,q) dp dq . 
\end{align*}
We bound    $|\Psi|\lesssim_c \vartheta^2\otimes \vartheta^2 \otimes \vartheta^2\otimes  \vartheta^2$  and apply Lemma \ref{lemma:a-size}. This yields
\begin{align*} 
|L|\lesssim_c \Big( \sum_{S\in \mathcal{L}(\T)}|S| + 1  \Big)    \lesssim 1.
\end{align*}
The last estimate follows  since  the leaves of $\T$ partition the root $R_\T$. 

Now we consider the right-hand side of \eqref{eqn:ftc}, which after convolving it with $\F$, integrating over $T_k$  and summing in $k\in \Z$ results in 
\begin{align*} 
R:= \sum_{k\in \Z} \int_{2^{k-1}}^{2^k}\int_{T_k}  \F((F_j)_{j\in J}) * (-t\partial_t   
[\Psi]_t)(p,q,p,q)  dp dq\frac{dt}{t},
\end{align*}
where $J:=\{1,2,3,4\}$.  
First we show that up to a controllable error, we may suppose that the functions $F_j$ are supported on $T_k$. 
For $j\in J$ we write $F_j= F_j\1_{T_k} + F_j\1_{T_k^c}$.  Then 
$$ R = M+E,$$ 
where the main  term is defined as 
\begin{align*}
M:=   \sum_{k\in \Z} \int_{2^{k-1}}^{2^k}\int_{T_k}  \F((F_j\1_{T_k})_{j\in J})* (-t\partial_t   
[\Psi]_t)(p,q,p,q)  dp dq\frac{dt}{t}
\end{align*}
and  the  error term  equals
\begin{align*}
E :=   \sum_{((X_{j,k})_{k\in \Z})_{j\in J}}    \sum_{k\in \Z}  \int_{2^{k-1}}^{2^k} \int_{T_k}  \F((F_j\1_{X_{j,k}})_{j\in J})* (-t\partial_t   
[\Psi]_t)(p,q,p,q)  dp dq\frac{dt}{t}, 
\end{align*}
where the outer summation is over $((X_{j,k})_{k\in \Z})_{j\in J}\in \{T,T^c\}^4\setminus \{(T,T,T,T)\}$ for $T:=(T_k)_{k\in\Z}$, $T^c:=(T_k^c)_{k\in\Z}$.

To treat   $E$ we expand
$-t\partial_t   
[\Psi]_t = -t\partial_t([\rho_1]_t \otimes [\rho_2]_t \otimes [\rho_1]_t \otimes [\rho_2]_t ) $ and use the chain rule, which results in  four terms. By symmetry we consider only $-t\partial_t([\rho_1]_t)\otimes [\rho_2]_t \otimes [\rho_1]_t\otimes [\rho_2]_t$, on which we use the identity 
\begin{align}\label{identity:chain}
 -t\partial_t [\rho_1]_t= -t\partial_t\Big (\frac{1}{t}\rho_1\Big (\frac{x}{t} \Big )\Big ) =  \frac{1}{t}\rho_1\Big (\frac{x}{t} \Big ) + \frac{1}{t} \frac{x}{t} \rho_1'\Big (\frac{x}{t} \Big ).
\end{align} and  bound   the right-hand side of \eqref{identity:chain} by $\lesssim_c [\vartheta^2]_t$.  
 This gives $|t\partial_t[\Psi]_t|\lesssim_c  [\vartheta^2\otimes \vartheta^2\otimes \vartheta^2\otimes \vartheta^2]_t $.  
 By Lemma \ref{lemma:error-Tc} we then    have  $|E| \lesssim_c 1$.

To estimate   $M$ we  expand the convolution  and interchange the order of integration such that the integration in $(p,q)$
 becomes the innermost. For now we consider only this innermost integral, which we write in the form  
 $$\int_{ T_k}-t\partial_t \Big( \Big  (  [\rho_{1}]_t(p-x)  [\rho_{1}]_t(p-x') \Big ) \Big ( [\rho_{2}]_t(q-y)     [ \rho_{2}]_t(q-y') \Big )\Big)  dpdq.$$
Deriving the product of $[\rho_{1}]_t(p-x)  [\rho_{1}]_t(p-x')$ and $[\rho_{2}]_t(q-y)     [ \rho_{2}]_t(q-y')$   yields two terms.
Using Fubini and  moving the derivative outside the integral we arrive at 
\begin{align}\label{term:leibnitz-1}
& \sum_{Q\in \mathcal{I}_k}\Big( -t\partial_t \int_{T_{Q,1}}  
[\rho_{1}]_t(p-x)[\rho_{1}]_t(p-x')   dp   \Big  )\int_{Q} [\rho_{2}]_t(q-y)[\rho_{2}]_t(q-y') dq  \\
\label{term:leibnitz-2}
+ &  \sum_{P\in \mathcal{I}_k}   \int_{P} 
[\rho_{1}]_t(p-x)[\rho_{1}]_t(p-x') dp \,\Big (-t\partial_t\int_{T_{P,2}} [\rho_{2}]_t(q-y)[\rho_{2}]_t(q-y')   dq \Big ),
\end{align}
where  for a dyadic interval $Q$ we denote $T_{Q,1} := \cup_{P:P\times Q \in \T} P$ and $T_{P,2}$ is defined analogously. 
As both parts are treated in the same way, 
we  further investigate only \eqref{term:leibnitz-1}. 

The  identity \eqref{eqn:FT} implies  
\begin{align*}
 -t \partial_t  \int_{\R}  
[\rho_{1}]_t(p-x)[\rho_{1}]_t(p-x') dp =  \int_{\R}  
[\sigma_{1}]_t(p-x)[\sigma_{1}]_t(p-x') dp,
\end{align*}
which can be seen by an application of the inverse Fourier transform on \eqref{eqn:FT}. 
Hence, 
\begin{align}\nonumber
& -  t\partial_t  \int_{T_{Q,1}}  
[\rho_{1}]_t(p-x)[\rho_{1}]_t(p-x')   dp =   \int_{T_{Q,1}}  
[\sigma_{1}]_t(p-x)[\sigma_{1}]_t(p-x') dp + b_1,
\end{align}
where  $b_1$ is the boundary portion
\begin{align*} 
 b_1:=  
   \int_{\R\setminus T_{Q,1}}  
[\sigma_{1}]_t(p-x)[\sigma_{1}]_t(p-x') dp 
+   t\partial_t  \int_{\R\setminus T_{Q,1}}  
[\rho_{1}]_t(p-x)[\rho_{1}]_t(p-x') dp.
\end{align*}
Therefore we have
\begin{align}   \label{eqn:sum-of-thetas}
M  =  \Big ( \sum_{k\in \Z} \Theta_{\sigma_1,\rho_2}^{\T_{k}}((F_j\1_{T_k})_{j\in J}) + \Theta_{\rho_1,\sigma_2}^{\T_{k}}((F_j\1_{T_k})_{j\in J}) \Big ) + B_1 + B_2,
\end{align} 
where the boundary term $B_1$ emerges from $b_1$ and equals
\begin{align*}
 B_1 := & \sum_{k\in \Z} \int_{2^{k-1}}^{2^k} \sum_{Q\in \mathcal{I}_k} \int_Q\int_{\R\setminus T_{Q,1}}  
 \F((F_j\1_{T_k})_{j\in J})(x,y,x',y')\\
 &    \Big( [\sigma_{1}]_t(p-x)  [\sigma_{1}]_t(p-x')) +   t\partial_t \big ([\rho_{1}]_t(p-x)  [\rho_{1}]_t(p-x') \Big)  [\rho_{2}]_t(q-x)  [\rho_{2}]_t(q-x')\\
&  \hspace{10.2cm}   dxdydx'dy'  dp dq\frac{dt}{t}. 
\end{align*}
The   boundary term $B_2$  arises from the treatment of  \eqref{term:leibnitz-2} and is analogous to $B_1$ with   $(\sigma_1,\rho_2)$ replaced by $(\rho_1,\sigma_2)$. For   $B_1,B_2$ we  derive by $t$ using \eqref{identity:chain} and dominate the resulting functions by $\lesssim_c \vartheta^2$. Note that
\begin{align*} 
 | B_1 + B_2|\lesssim_c \sum_{k\in \Z}\Theta^{\T_k^c}_{\vartheta^2,\vartheta^2}((F_j\1_{T_k})_{j\in J}) \lesssim 1,
\end{align*}
where the last inequality follows by Lemma \ref{lemma:bdry-term}.

Summarizing,  since  $L = R=M +   E$,   using \eqref{eqn:sum-of-thetas}   yields  the identity
\begin{align}\label{tel:almost-final}
L = \Big ( \sum_{k\in \Z} \Theta_{\sigma_1,\rho_2}^{\T_{k}}((F_j\1_{T_k})_{j\in J}) + \Theta_{\rho_1,\sigma_2}^{\T_{k}}((F_j\1_{T_k})_{j\in J}) \Big ) + B_1+B_2+E. 
\end{align}
Proposition \ref{prop:tel}  now follows by writing  
\begin{align*}
\Theta_{\rho_1,\sigma_2}^{\T}((F_j)_{j\in J})  +  \Theta_{\sigma_1,\rho_2}^{\T}((F_j)_{j\in J}) 
\end{align*}
in the form  
\begin{align*}
 & \sum_{k\in \Z} \Theta_{\sigma_1,\rho_2}^{\T_{k}}((F_j\1_{T_k})_{j\in J}) +   \Theta_{\rho_1,\sigma_2}^{\T_{k}}((F_j\1_{T_k})_{j\in J}) \\
  +    \sum_{((X_{j,k})_{k\in \Z})_{j\in J}}  &  \sum_{k\in \Z}   \Theta_{\rho_1,\sigma_2}^{\T_{k}}((F_j\1_{X_{j,k}})_{j\in J}) +      \Theta_{\sigma_1,\rho_2}^{\T_{k}}((F_j\1_{X_{j,k}})_{j\in J}), 
\end{align*}
where in the second line, the outer sum   runs over $((X_{j,k})_{k\in \Z})_{j\in J} \in \{T,T^c\}^4\setminus \{(T,T,T,T)\}$ for $T$ as above. 
Using   \eqref{tel:almost-final} together with  
$$|L -  B_1- B_2-E| \lesssim_c 1$$ and evoking Lemma  \ref{lemma:error-Tc} two more times finally   yields \eqref{eqn:tel-to-establish}.
\end{proof}

\section{Tree estimate}
In this section  we derive  an estimate for    a  quadrisublinear variant  of $\Lambda_{\varphi,\psi}^{N}$  restricted to  $\Omega_\T$ for a convex tree $\T$. This form is given by
\begin{align*}
\widetilde{\Theta}^\T_{\varphi,\psi}(F_1,F_2,F_3,F_4):= \int_{\Omega_\T} \big  |   \F  * [{\varphi^{(u)}}\otimes {\psi^{(v)}} \otimes {\varphi^{(-u)}} \otimes{\psi^{(-v)}}]_t(p,q,p,q) \big |dpdq\frac{dt}{t}.
\end{align*}
It can also be recognized as a quadrisublinear version of ${\Theta}^\T_{\varphi^{(u)},\psi^{(v)},\varphi^{(-u)},\psi^{(-v)}}$.
\begin{proposition}\label{prop:single-tree}
We have the estimate
\begin{align}\label{def:loc-vers-lambda}
\widetilde{\Theta}^\T_{\varphi,\psi}(F_1,F_2,F_3,F_4)\lesssim |R_{\mathcal{\T}}|\prod_{j=1}^4M(F_j,\T).
\end{align}
\end{proposition}
The proof of Proposition \ref{prop:single-tree} proceeds in a very similar  way as the proof  of the $\L^4$ bound \eqref{def:main}.
 Besides replacing \cite[Lemma 3]{pd:L4} 
 with   Proposition \ref{prop:tel}, the only   modification is the choice of a faster  decaying superposition of the Gaussian exponential functions \eqref{sup-gauss}. For completeness we summarize all steps of the proof,   interested readers  are referred to \cite{pd:L4}. 
\begin{proof}
By homogeneity and scale-invariance we may   suppose $M(F_j,\T)=1$  and   $|R_{\mathcal{\T}}|=1$.
First we expand  the left-hand side of \eqref{def:loc-vers-lambda}  and use the triangle inequality to arrive at
\begin{align*}
 \int_{\Omega_\T} \int_{\R^2} \left | \int_{\R}F_1(x,y)F_2(x',y)[\psi^{(v)}]_{t}(q-y)  dy \right. \left. \int_{\R} F_3(x',y')F_4(x,y') [\psi^{(-v)}]_{t}(q-y') dy' \right| & \\
   [|\varphi^{(u)}|]_t(p-x) [|\varphi^{(-u)}|]_t(p-x') dxdx'dpdq\frac{dt}{t}. &
\end{align*}
By an application of the Cauchy-Schwarz inequality, this is bounded by
\begin{align*}
{\Theta}^{\T}_{|\varphi^{(u)}|,\psi^{(v)},|\varphi^{(-u)}|,\psi^{(v)}}(F_1,F_2,F_2,F_1)^{1/2} 
 {\Theta}^{\T}_{|\varphi^{(u)}|,\psi^{(-v)},|\varphi^{(-u)}| ,\psi^{(-v)}}(F_4,F_3,F_3,F_4)^{1/2}.
\end{align*}

As both terms are treated analogously, we consider  the first one only. We shall now apply the telescoping identity, for which we dominate $\varphi^{(\pm u)}$  with a superposition of   Gaussians. Denote the $\mathrm{L}^1$-normalized Gaussian exponential function rescaled by $\alpha >0$ by
\begin{align*}
g_{\alpha}(x)&:=\frac{1}{\sqrt{\pi}\alpha} e^{-\left( \frac{x}{\alpha}\right)^2}.
\end{align*}
Consider the superposition of the functions $g_\alpha$ given by
\begin{align} \label{sup-gauss}
\Phi(x):=\int_1^\infty \frac{1}{\alpha^{21}} e^{-\left( \frac{x}{\alpha}\right)^2} d\alpha = \frac{1}{\sqrt{\pi}} \int_1^\infty \frac{1}{\alpha^{20}} g_\alpha(x) d\alpha. 
\end{align}
For large $x$ we have
 $\Phi(x) \sim x^{20}$, which   can be seen by the change of variables $\alpha' = (x/\alpha)^2$ and by inductive integration by parts.  
The power of $\alpha$ is now larger as in \cite{pd:L4}, as due to Proposition \ref{prop:tel} we need control over higher Schwartz seminorms of $g_\alpha$.

Since $\varphi^{(\pm u)}\in \S(\R^2)$, we can bound it by $\Phi$ times a  positive constant, which is uniform  in $u$. 
By positivity of
\begin{align*}
{\Theta}^{\T}_{|\varphi^{(u)}|,\psi^{(v)},|\varphi^{(-u)}|,\psi^{(v)}}(F_1,F_2,F_2,F_1) =  \int_{\Omega_\T} \int_{\R^2}  \Big ( \int_{\R} F_1(x,y)F_2(x',y)  [\psi^{(v)}]_{t}(q-y)  dy \Big )^2 &\\
  [|\varphi^{(u)}|]_t(p-x) \, [|\varphi^{(u)}|]_t(p-x') dxdx'dpdq\frac{dt}{t},&
\end{align*}
we can estimate    this term up to a constant by    
\begin{align*}
  \int_1^\infty\int_1^\infty \Theta^{\T}_{g_\alpha,\psi^{(v)},g_\beta ,\psi^{(v)}}(F_1,F_2,F_2,F_1) 
\frac{d\alpha}{\alpha^{20}} \frac{d\beta}{\beta^{20}}.
\end{align*} 
We split the integration into the regions $\alpha\geq \beta$ and $\alpha < \beta$.
By symmetry it suffices to estimate the region $\alpha\geq \beta$ only, on which 
$\beta g_\beta \leq \alpha g_\alpha$
for $\alpha,\beta \geq 1$. This 
leaves us with 
$$\int_1^\infty \Theta^{\T}_{g_\alpha,\psi^{(v)}}(F_1,F_2,F_2,F_1) \frac{d\alpha}{\alpha^{19}}.$$

Now we are ready to apply Proposition \ref{prop:tel} with 
$(\rho_1,\sigma_1)=(g_\alpha,h_\alpha)$ and $(\rho_1,\sigma_2)=(\phi,\psi^{(v)})$, where $h_{\alpha}(x):={\alpha}(g_{\alpha})'(x)$ and 
\begin{align*}
\widehat{\phi}(\xi):=\left( \int_{\xi}^\infty | \widehat{\psi^{(v)}}(\tau) |^2 \frac{d\tau}{\tau} \right)^{1/2},
\end{align*}
which is a Schwartz function by our condition on $\psi$.
Proposition \ref{prop:tel} yields
\begin{align} \label{after1tel} 
{\Theta}^{\T}_{g_\alpha,\psi^{(v)}}(F_1,F_2,F_2,F_1)   & \lesssim - {\Theta}^{\T}_{h_\alpha,{\phi}}(F_1,F_2,F_2,F_1) +  c
\end{align}
with $c= \|g_\alpha\|^2 \|\psi^{(v)}\|^2 + \|\phi\|^2 \|h_\alpha\|^2 + \|g_\alpha\|^2 \|\phi\|^2 \lesssim \alpha^{16}$. 
Thus it remains to estimate the form on the right-hand side of  \eqref{after1tel}. 

In the second iteration of the procedure we 
 bound
$|{\Theta}_{h_\alpha,\phi}^{\T}(F_1,F_2,F_2,F_1)| $ by
\begin{align*}
  \int_{\Omega_\T} \int_{\R^2}  \left | \int_{\R}F_1(x,y)F_1(x,y')[h_{\alpha}]_t(p-x)  dx \right.
 \left.\int_{\R} F_2(x',y')F_2(x',y) [h_{\alpha}]_t(p-x') dx' \right| & \\
 [|\phi|]_{t}(q-y)[|\phi|]_{t}(q-y') dydy'dpdq\frac{dt}{t}. &
\end{align*} 
Again we apply the Cauchy-Schwarz inequality and arrive to
\begin{align*}
|{\Theta}^{\T}_{h_\alpha,{\phi}}(F_1,F_2,F_2,F_1)| \leq {\Theta}^{\T}_{h_\alpha,|\phi|}(F_1,F_1,F_1,F_1)^{1/2}
{\Theta}^{\T}_{{h}_\alpha,|\phi|}(F_2,F_2,F_2,F_2)^{1/2}
\end{align*}
Dominating  the  rapidly decaying   $|\phi|$ by a positive constant times $\Phi$  gives 
\begin{align*}
{\Theta}^{\T}_{h_\alpha,|\phi|}(F_1,F_1,F_1,F_1) \lesssim 
\int_1^\infty \int_1^\infty \Theta^{\T}_{h_\alpha,g_\gamma,{h}_\alpha,g_\delta}(F_1,F_1,F_1,F_1)\frac{d\gamma}{\gamma^{20}} \frac{d\delta}{\delta^{20}}.
\end{align*}
As before, by symmetry  this reduces to having to estimate
$$\int_1^\infty {\Theta}^{\T}_{h_\alpha,g_\gamma}(F_1,F_1,F_1,F_1) \frac{d\gamma}{\gamma^{19}}.$$
Now we apply Proposition \ref{prop:tel} to the pairs $(\rho_1,\sigma_1)=(g_\alpha,h_\alpha)$ and $(\rho_2,\sigma_2)=(g_\gamma,h_\gamma)$, giving
\begin{align*}
{\Theta}^{\T}_{h_\alpha,g_\gamma}(F_1,F_1,F_1,F_1) & \lesssim    - {\Theta}^{\T}_{g_\alpha,h_\gamma}(F_1,F_1,F_1,F_1) +  c
\end{align*}
with $c= \|g_\alpha\|^2 \|h_\gamma\|^2 + \|g_\gamma\|^2 \|h_\alpha\|^2 + \|g_\alpha\|^2 \|g_\gamma\|^2 \lesssim \alpha^{16}\gamma^{16}$. 
Finally observe that 
\begin{align*}
 {\Theta}^{\T}_{g_\alpha,h_\gamma}(F_1,F_1,F_1,F_1)\geq 0, 
\end{align*}
which can be seen by writing it as an integral of a square   multiplied with $g_\alpha \geq 0$. 
Thus, 
$${\Theta}^{\T}_{h_\alpha,g_\gamma}(F_1,F_1,F_1,F_1) \leq 1. $$ 
This concludes the proof in view of our normalization.
\end{proof}

\section{Completing the proof of Theorem \ref{thm:GRTmain}} 
Now we are ready to establish the restricted type estimate from Theorem \ref{thm:GRTmain}. We adapt the approach of \cite{th:wpa} and also rely on \cite{th:phd}. 
\begin{proof}[Proof of Theorem \ref{thm:GRTmain}]
First note that by quadrilinearity of $\Lambda_{\varphi,\psi}^N$ it suffices to prove the theorem for positive functions $F_j$, as otherwise we split them into real and imaginary, positive and negative parts.

For $j=1,2,3,4$ let  $\alpha_j$ be  such that $-1/2 \leq  \alpha_j \leq 1/2$  and $\alpha_1+\alpha_2+\alpha_3+\alpha_4 = 1$.  For each $j$ let $E_j \subseteq \R^2$ be   measurable. Without loss of generality we may assume    $|E_1|$ is maximal among the $|E_j|$.  
Note that for $a=2^k$ we have the scaling identity
\begin{align*} 
\Lambda^{N}_{\varphi,\psi}(F_1,F_2,F_3,F_4) = a^2\Lambda^{N/a}_{\varphi,\psi}(F_1(a\cdot ),F_2(a\cdot ),F_3(a\cdot ),F_4(a\cdot ) ).
\end{align*}  
Since our bound will be independent of $N$, by   $\sum_j \alpha_j=1$  we may then  suppose $1\leq |E_1|\leq 4$. All squares  which we   consider in this section are assumed to  have their side-lengths in the interval $[2^{-N},2^N]$.

For $F$ on $\R^2$ we denote the quadratic Hardy-Littlewood maximal function by
$$\M(F) :=\sup_{S  }\ \Big ( \frac{1}{|S|}\int_S F^2 \Big )^{1/2}\1_S , $$
where the supremum is taken over all (not necessarily dyadic) squares in $\R^2$ with sides parallel to the coordinate axes. From now on, by the word ''average'' we will always mean the second power average as in the definition of $\M(F)$. Define the exceptional set
$$H:=\bigcup_{j=1}^4 \{\M(|E_j|^{-1/2}\mathbf{1}_{E_j}) >2^{10}\}.$$
By the Hardy-Littlewood maximal theorem we have $|H|\leq 1/18$. 
Let $\mathcal{R}$ be the set of all dyadic squares $R\subseteq H$ which are maximal with respect to   set inclusion. Denote by $3R$ the square with the same center as $R$ but with three times the sidelength of $R$. We
set ${E}'_1:=E_1\setminus \cup_{R\in \mathcal{R}}3R$. Then $2|{E}'_1| \geq  |E_1|$.

Suppose we are given four  functions $F_j$  with   $|F_j|\leq \1_{E_j}$ for all $j$ and  
$|F_1|\leq \1_{E_1'}$. Since $\alpha_j\leq 1/2$ and $|E_1|\leq 4$, it suffices to prove 
$$|{\Lambda}^N_{\varphi,\psi}(F_1,F_2,F_3,F_4)| \lesssim |E_1|^{1/2}|E_2|^{1/2}|E_3|^{1/2}|E_4|^{1/2}.$$
If we set $G_j:=|E_j|^{-1/2}F_j$, then the inequality we need to establish reads
$$|{\Lambda}^N_{\varphi,\psi}(G_1,G_2,G_3,G_4)| \lesssim 1.$$
Observe that $\|G_j\|_{\L^2(\R^2)} \leq 1$ for all $j$.

We split $\R^2\times[2^{-N},2^N] $ into the regions $\Omega_{\{S\}}=S\times [\ell(S)/2,\ell(S)]$, $S\in \D$,  and consider the cases $S\subseteq H$ and $S\not\subseteq H$.  
By the triangle inequality we   estimate
\begin{align*}
 |{\Lambda^N_{\varphi,\psi}} | \leq
\sum_{S\subseteq H}\widetilde{\Theta}^{\{S\}}_{\varphi,\psi} + \sum_{S\not\subseteq H}\widetilde{\Theta}^{\{S\}}_{\varphi,\psi}. 
 \end{align*}

First we consider the sum over   $S\not\subseteq H$. 
For $k \in \Z$ let $\S_k$   be the set of all dyadic squares $S$ 
for which  
\begin{align*}
 2^{k-1} < \max_{j\in \{1,2,3,4\}} \;\sup_{S'\supseteq S}  \Big( \frac{1}{|S'|}  \int_{S'} G_j^2 \Big )^{1/2}  \leq 2^{k}.
\end{align*}
The supremum is taken over all (not necessarily dyadic) squares $S'\supseteq S$ in $\R^2$   with sides parallel to the coordinate axes.
Denote by $\mathcal{R}_k$ the collection of the maximal squares in $\S_k$ with respect to   set inclusion. 
For  $R\in \mathcal{R}_k$ we define 
$$\T_R:=\{S\in \S_k : S \subseteq R\},$$
which  is a   convex tree with the root $R$. Convexity follows from monotonicity of the supremum. 
By construction, if $S\not\subseteq H$,  for each $j$ the average of $|E_j|^{-1/2}\1_{E_j}$ over $S$ is no greater than $2^{10}$. Thus, the same holds for the average of $G_j$ over $S$. 
Therefore,   $$\{S: S\not\subseteq H\}\subseteq \bigcup_{k\leq 10}\S_{k}$$ and  we can split the summation as
\begin{align*}
\sum_{S\not\subseteq H} \widetilde{\Theta}_{\varphi,\psi}^{\{S\}} \leq   \sum_{k\leq 10}\sum_{R\in \mathcal{R}_k} \sum_{S\in \T_{R}} \widetilde{\Theta}_{\varphi,\psi}^{\{S\}}  = \sum_{k\leq 10}\sum_{R\in \mathcal{R}_k} \widetilde{\Theta}_{\varphi,\psi}^{\T_{R}} .
\end{align*}
For the forms on the right-hand side   we have  by Proposition \ref{prop:single-tree} that
\begin{align} \label{eqn:ste}
\widetilde{\Theta}_{\varphi,\psi}^{\T_{R}}(G_1,G_2,G_3,G_4)\lesssim  |R| \prod_{j=1}^4 M(G_j, \T_{R}).
\end{align}

To estimate  the right-hand side of \eqref{eqn:ste} we   discretize the function $\theta$ by a standard approximation with characteristic functions of balls of radius at least $1$.  We now sketch the required argument. 
Denote by $B_r$   the ball of radius $r$ centered at $0$ in $\R^2$. We  write  
\begin{align*}G_j^2*[\theta]_t & =  G_j^2*[\theta\1_{B_{1}}]_t + G_j^2*[\theta\1_{B_1^c}]_t.
\end{align*}  
Let $(p,q,t)\in  {S}\times [\ell(S)/2, \ell(S)] \subseteq    {\Omega_{\T_R}}$ and assume $(p,q)=0$.
On $B_1$ we have
\begin{align}\label{bound:on-b}
G_j^2*[\theta\1_{B_{1}}]_t(0) \lesssim 
\|\theta\|_{\L^\infty(\R^2)}  \frac{1}{(2t)^2} \int_{[-t,t]^2 } G_j^2 \lesssim \frac{1}{(2\ell(S))^2} \int_{[-\ell(S),\ell(S)]^2 } G_j^2 
\lesssim 2^{2k}.
\end{align}
For the part on $B_1^c$ we    consider the  function $ \theta\1_{B_1^c} + \frac{1}{2}\1_{B_1}$. It dominates $\theta\1_{B_1^c}$,  is positive and radially decreasing. Therefore it can be  approximated pointwise by 
a monotonously increasing sequence of
simple functions of the form  
\begin{align*}
E=\sum_{i=1}^{n} a_i \mathbf{1}_{B_{r_i}},\;\, r_i \geq 1,\; a_i >0.
\end{align*} 
For $E$ we  have, using $t\sim \ell(S)$, that
\begin{align*}  G_j^2*[E]_t(0)   \lesssim \sum_{i=1}^n a_i |B_{r_i}| \frac{1}{(r_i\ell(S))^2} \int_{[-r_i\ell(S),r_i\ell(S)]^2 } G_j^2  \lesssim  \| \theta \|_{\L^1(\R^2)} 2^{2k}.
 \end{align*}
This implies the estimate  
\begin{align}\label{bound:on-bc}
G_j^2*[\theta\1_{B^c_{1}}]_t(0) \lesssim 2^{2k}.
\end{align}
By a translation argument, the same bound holds at any $(p,q,t)\in \Omega_{\T_R}$. 
Therefore, by \eqref{bound:on-b} and \eqref{bound:on-bc}, we have $M(G_j,\T_R)\lesssim 2^k$ for each $j$  and hence
\begin{align}\label{eqn:not-in-h}
\sum_{S\not\subseteq H}\widetilde{\Theta}_{\varphi,\psi}^{\{S\}}(G_1,G_2,G_3,G_4) \lesssim \sum_{k\leq 10} 2^{4k}\sum_{R\in \mathcal{R}_k}|R|.
\end{align}

It remains to sum up the right-hand side of the last display. Since for $R\in \mathcal{R}_k$ there is an index $j$ such that on $R$ we have $\M(G_{j}) >2^{k-1}$, by maximality of the squares in $\mathcal{R}_k$   
\begin{align*}
\sum_{R\in \mathcal{R}_k} |R| = \Big | \bigcup_{R\in \mathcal{R}_k} R \;\Big | \leq  \sum_{j=1}^4|\{\M(G_{j}) > 2^{k-1}\}|.
\end{align*}
 By the Hardy-Littlewood maximal theorem and $\|G_{j}\|_{\L^2(\R^2)}\leq 1$, for each $j$ we have  $|\{\M(G_{j}) > 2^{k-1}\}|\lesssim 2^{-2k}$.
Thus, \eqref{eqn:not-in-h} is up to a constant dominated by
$$  \sum_{k\leq 10} 2^{2k} \lesssim 1.$$
This establishes the desired estimate for $S\not\subseteq H$.

Now consider the sum over all dyadic squares   $S$ contained in $H$. 
Every $S\subseteq H$  is  contained in one maximal dyadic square $R  \in \mathcal{R}$.
Let $\S_{R,k}$ be the set of dyadic squares $S$ which are $k$  generations below $R \in \mathcal{R}$. That is, $2^k\ell(S)=\ell(R)$. We split
\begin{align*}
\sum_{S\subseteq H} \widetilde{\Theta}_{\varphi,\psi}^{\{S\}}& = \sum_{R\in \mathcal{R}}\sum_{ k\geq 0}\sum_{S\in \S_{R,k}} \widetilde{\Theta}_{\varphi,\psi}^{\{S\}}. 
\end{align*}
For $S\in \S_{R,k}$ we   expand $\widetilde{\Theta}_{\varphi,\psi}^{\{S\}}(G_1,G_2,G_3,G_4)$ and estimate $|\varphi^{(u)}|,|\psi^{(v)}| \lesssim \vartheta^4$ to arrive at   
\begin{align}\label{term:s-in-h}\nonumber
\int_{\ell(S)/2}^{\ell(S)}   \int_S \int_{\R^4} \F(G_1,G_2,G_3,G_4)(x,y,x', y') [\vartheta\otimes \vartheta\otimes\vartheta\otimes\vartheta]_t (p-x,q-y,p-x',q- &y')  \\ 
  \theta^2(t^{-1}(p-x,q-y))\, dxdy dx'dy'dpdq \frac{dt}{t}.&
\end{align}
Since $G_1$ is supported on the complement of  $3R$,  we have $|(p,q)-(x,y)|\geq \ell(R)$ for $(p,q) \in S$. We also have $\ell(R)=2^k\ell(S)\sim 2^kt$, therefore  $\theta^2(t^{-1}(p-x,q-y)) \lesssim 2^{-8k}$. 
Applying Lemma \ref{lemma:a-size}, the term \eqref{term:s-in-h} is then up to a constant dominated by 
\begin{align*} 
2^{-8k}|S|\prod_{j=1}^4M(G_j,\{S\}).
\end{align*}
Denote by $R'$ the parent of $R$. For  each $j$ we have 
$$M(G_j,\{S\}) \lesssim 2^k   M(G_j,\{R'\})  \lesssim 2^k. $$
The last inequality follows by the same  approximation argument as before and using that the averages of $G_j$ over squares containing $R'$ are less than $2^{10}$,
which is true by maximality of $R$.    This establishes
\begin{align*}
\sum_{S\subseteq H} \widetilde{\Theta}_{\varphi,\psi}^{\{S\}}(G_1,G_2,G_3,G_4) \lesssim  \sum_{R\in \mathcal{R}}\sum_{ k\geq 0}\sum_{S\in \S_{R,k}} 2^{-4k}|S|. 
\end{align*}
Since $\sum_{S\in \S_{R,k}}|S| \leq |R|$, the last display is estimated by
 \begin{align*}
 \sum_{R\in \mathcal{R}}|R|\sum_{k\geq 0} 2^{-4k}  \lesssim |H| \lesssim 1.
 \end{align*}
 For the second to last inequality we summed the geometric series  and used  disjointness of $R\in \mathcal{R}$. In the last step we used $|H|\leq 1/2$. 
\end{proof}

\end{document}